\documentclass[11 pt]{article}

\usepackage[latin1]{inputenc}
\usepackage{amsmath,amsfonts,hyperref,amsthm, enumitem, graphicx}
\usepackage[textsize=footnotesize,color=green!40]{todonotes}
\usepackage{fullpage}
\usepackage{authblk}
\usepackage{changes}
\usepackage{bbm}
\usepackage{wrapfig}
\usepackage{caption}

\newtheorem{theorem}{Theorem}
\newtheorem{lemma}[theorem]{Lemma}
\newtheorem{proposition}[theorem]{Proposition}
\newtheorem{corollary}[theorem]{Corollary}
\newtheorem{conjecture}[theorem]{Conjecture}
\theoremstyle{definition}
\newtheorem*{definition*}{Definition}

\newtheorem*{claim}{Claim}

\newcommand{\cE}{\mathcal{E}}
\newcommand{\cF}{\mathcal{F}}
\newcommand{\cG}{\mathcal{G}}
\newcommand{\cH}{\mathcal{H}}

\newcommand{\bE}{\mathbb{E}}
\newcommand{\bZ}{\mathbb{Z}}

\newcommand{\sm}{\setminus}
\newcommand{\cmax}{m_*}

    \makeatletter
    \let\@fnsymbol\@arabic
    \makeatother

\newcommand{\eps}{\varepsilon}

\newcommand{\bP}{\mathbb{P}}

\newcommand{\cM}{\mathcal{M}}
\newcommand{\bN}{\mathbb{N}}

\newcommand{\s}{[s]}
\newcommand{\whp}[1]{1-o_{#1}(1)}
\newcommand{\defeq}{\mathrel{\mathop:}=}
\newcommand{\parts}[2]{#1_1^{#2} \cup #1_2^{#2}}
\newtheoremstyle{case}{}{}{}{}{}{:}{ }{}
\theoremstyle{case}
\newtheorem{case}{Case}

\newcommand{\N}{N}

\newcommand{\comment}[1]{}

\begin{document}
\title{Rainbow matchings in Dirac bipartite graphs}
\author{Matthew Coulson\thanks{School of Mathematics, University of Birmingham, UK. Email: mjc685@bham.ac.uk.}~ and Guillem Perarnau\thanks{School of Mathematics, University of Birmingham, UK. Email: g.perarnau@bham.ac.uk.}}

\maketitle

\begin{abstract}
We show the existence of rainbow perfect matchings in $\mu n$-bounded edge colourings of Dirac bipartite graphs, for a sufficiently small $\mu>0$. As an application of our results, we obtain several results on the existence of rainbow $k$-factors in Dirac graphs and rainbow spanning subgraphs of bounded maximum degree on graphs with large minimum degree.
\end{abstract}

\section{Introduction}


An $n \times n$ array of symbols where each symbol occurs precisely one in each row and column is called a \emph{Latin square} of order $n$.
A \emph{partial transversal} of size $k$ in a Latin square is a set of cells, including at most one from each row and each column that contains $k$ distinct symbols.
The question of finding the largest transversal in an arbitrary Latin square has attracted considerable attention.
There exist Latin squares, such as the addition table of $\bZ_n$ for even $n$, whose largest transversal has size $n-1$~\cite{edelta,wwdelta}.
It has been conjectured that this is the worst case. 
\begin{conjecture}[Ryser, Brualdi, Stein~\cite{brualdi1991combinatorial,r1967,s1975}]
\label{RBS}
Every Latin square of order $n$ contains a partial transversal of size at least $n-1$.
\end{conjecture}
The best known lower bound is due to Hatami and Shor~\cite{hs2008}, who showed that every Latin square of order $n$ has a partial transversal of size $n-O(\log^2(n))$. 

There is a one-to-one correspondence between Latin squares $L=(L_{ij})_{i,j\in [n]}$ of order $n$ and proper edge colourings of the complete bipartite graph $K_{n,n}$ on $2n$ vertices; simply, assign colour $L_{ij}$ to the edge $a_ib_j$, where $A=\{a_1,\dots, a_n\}$ and $B=\{b_1,\dots, b_n\}$ are the stable sets of $K_{n,n}$. 
Given a graph $G$ and an edge colouring $\chi:E(G)\to \bN$ of $G$, the subgraph $H\subseteq G$ is \emph{rainbow} if for every $c\in\bN$, $|\chi^{-1}(c)\cap E(H)|\leq 1$. Under the above correspondence,  a partial transversal of size $k$ in a Latin square is equivalent to a partial rainbow matching of size $k$.


One can extend the problem to edge colourings of $K_{n,n}$ that satisfy a milder condition. An edge colouring is \emph{$k$-bounded} if $|\chi^{-1}(c)|\leq k$ for every $c\in \bN$. Stein~\cite{s1975} conjectured that Conjecture~\ref{RBS} still holds for $n$-bounded edge colourings. This was very recently disproved by Pokrovskiy and Sudakov~\cite{bs2017}. However, positive results can be obtained if the size of each colour class is small enough.
\begin{theorem}[Erd\H{o}s, Spencer~\cite{l4es}]
\label{ES4e}
Let $K_{n,n}$ be the complete bipartite graph on $2n$ vertices, then any $(n-1)/16$-bounded edge colouring of $K_{n,n}$ contains a rainbow perfect matching.
\end{theorem}

The goal of this paper is to obtain a sparse version of Theorem~\ref{ES4e}. A balanced bipartite graph $G$ contains a perfect matching if and only if $G$ satisfies Hall's condition. However, it is easy to see that Hall's condition is not a sufficient for the existence of a rainbow perfect matching if colour classes have linear size.  For example, consider a graph consisting of a perfect matching which trivially satisfies Hall's Condition but has no rainbow perfect matching unless each colour class has size $1$. Thus, we impose a stronger condition concerning the minimum degree of $G$. A \emph{Dirac bipartite graph} on $2n$ vertices is a balanced bipartite graph with minimum degree at least $n/2$. The main result of this paper shows the existence of rainbow perfect matchings in Dirac bipartite graphs.


\begin{theorem}
\label{mainthm}
There exist $\mu > 0$ and $n_0 \in \bN$ such that if $n \geq n_0$ and $G$ is a Dirac bipartite graph on $2n$ vertices, then any $\mu n$-bounded edge colouring of $G$ contains a rainbow perfect matching.
\end{theorem}

The proof of Theorem~\ref{mainthm} combines probabilistic and extremal ingredients. The main tool used to provide the existence of a rainbow matching is a lopsided version of the Lov\'{a}sz Local Lemma, which is standard in this setting.  One of the novelties of our approach is the estimation of conditional probabilities in the uniform distribution on the set of perfect matchings of a Dirac bipartite graph, via a switching argument (see Section~\ref{sec:c6}). However, this probability space often exhibits strong dependencies which limit the application of the local lemma.

In order to overcome this problem, in Section~\ref{sec:dico} we use a well-established dichotomy for Dirac bipartite graphs; either the graph has good expansion properties (\emph{robust expander}) or 
the graph consists of two (possibly unbalanced) very dense bipartite graphs of order roughly $n$ with few edges in-between (\emph{extremal graph}).   
The notion of robust expanders was first introduced by K{\"u}hn et al.~\cite{DDT} in the context of Hamiltonian digraphs (see also~\cite{kuhn2014hamilton}).
A local lemma based argument provides the existence of a rainbow perfect matching in robust expanders~(Section~\ref{sec:robexp}). However, this argument cannot be applied directly to extremal graphs. In Section~\ref{sec:ext}, we construct a rainbow perfect matching by selecting a partial matching in-between the two dense bipartite graphs that balances the remainder, followed by extending it into a rainbow perfect matching using similar arguments to the ones displayed previously. In Section~\ref{sec:final} we combine these two results, concluding that any Dirac bipartite graph with a $\mu n$-bounded edge colouring contains a rainbow perfect matching.

\medskip

Our result can be extended to a more general setting by slightly strengthening the minimum degree condition. A \emph{system of conflicts} for $E(G)$ is a set $\cF$ of unordered pairs of edges of $G$. If $\{ e, f \} \in \cF$ we say that $e$ and $f$ conflict and call $\{ e, f \}$ a \emph{conflict}.
A system of conflicts $\cF$ for $E(G)$ is \emph{$k$-bounded} if for each $e \in E(G)$, there are at most $k$ conflicts that contain $e$.

Given a graph $G$ and a system of conflicts $\cF$ for $E(G)$, the subgraph $H \subseteq G$ is \emph{$\cF$-conflict-free} if for each distinct $e,f \in E(H)$, we have  $\{ e, f \} \not \in \cF$.



Rainbow subgraphs correspond to conflict-free subgraphs of transitive systems of conflicts.
Given an edge colouring $\chi$ of $G$, we define the system of conflicts $\cF_{\chi}$ for $E(G)$ as follows
$$
\cF_{\chi} = \{ \{ e, f \}: e, f \in E(G) \text{ and } \chi(e) = \chi(f) \}
$$
Note that $\chi$ is $k$-bounded if and only $\cF_\chi$ is $(k-1)$-bounded. 

We obtain an asymptotic version of Theorem~\ref{mainthm} for bounded systems of conflicts.
\begin{theorem}
\label{conflicts}
For all $\eps > 0$ there exist $\mu > 0$ and $n_0 \in \bN$ such that if $n \geq n_0$ and $G$ is a balanced bipartite graph on $2n$ vertices with minimum degree $\delta(G) \geq (1/2 + \eps) n $, then any $\mu n$-bounded system of conflicts $\cF$ for $E(G)$ contains a conflict-free perfect matching.
\end{theorem}

Theorem~\ref{conflicts} follows as a corollary of the proof of Theorem~\ref{mainthm} for robust expanders (Section~\ref{sec:final}).
Section~\ref{sec:appl} contains two applications of Theorem~\ref{mainthm} and Theorem~\ref{conflicts}, providing the existence of rainbow $\Delta$-factors in Dirac graphs and of rainbow spanning subgraphs with bounded maximum degree in graphs with large minimum degree. We conclude the paper in Section~\ref{sec:fur_rem} with further remarks and open questions.

\section{Switching over \texorpdfstring{$6$}{6}-cycles}\label{sec:c6}
Our main tool to find conflict-free matchings is a $p$-Lopsided form of the Lov\'asz Local Lemma. In this section we introduce it and show how it will be applied.
\begin{definition*}
Let $\cE=\{E_1, E_2, \ldots, E_q\}$ be a collection of events. A graph $\cH$ with vertex set $[q]$ is a \emph{$p$-dependency graph for $\cE$} if for every $i\in [q]$ and every set $S \subseteq [q] \sm N_\cH[i]$ such that $\bP(\cap_{j \in S} E_j^c)>0$ we have
\begin{eqnarray}\label{eq:dep_cond}
\bP(E_i \vert \cap_{j \in S} E_j^c)& \leq p\;.
\end{eqnarray}
\end{definition*}

We will use the following version of the local lemma that uses $p$-dependency graphs.
\begin{lemma}[$p$-Lopsided Lov\'asz Local Lemma~\cite{l4es}]
\label{L4}
Let $\cE$ be a collection of events and let $\cH$ be a $p$-dependency graph for $\cE$. Let $d$ be the maximum degree of $\cH$.
If $4pd \leq 1$, then  
$$
\bP(\cap_{E\in \cE} E^c)\geq (1-2p)^{|\cE|}\;.
$$
\end{lemma}
While the statement in~\cite{l4es} does not specify the explicit lower bound, this is contained inside their proof.

The following notion will play a central role in showing the existence of conflict-free perfect matchings.
\begin{definition*}
Let $G = (A \cup B, E)$ be a balanced bipartite graph on $2n$ vertices with at least one perfect matching. Suppose $M$ is a perfect matching of $G$ and let $x=a_1b_1 \in M$. An edge $y=ab \in E(G)$ is \emph{$(x, M)$-switchable} if $y\notin M$ and the $6$-cycle $a_1b_1a_2 b a b_2$ is a subgraph of $G$, where $a_2b,ab_2\in M$.
\end{definition*}

The existence of many switchable edges in every perfect matching suffices to find a conflict-free perfect matching.
In the following lemma and in many of the subsequent results of this paper we use hierarchies of the form $1/n \ll \alpha \ll \beta \ll 1$.
We write $\alpha \ll \beta$ to mean that there exists an increasing function $f$ such that our result holds whenever $\alpha \leq f(\beta)$.
For simplicity we do not explicitly calculate such functions, however we could do so in principle should we so wish.

\begin{lemma}
\label{key}
Let $n\in\bN$ and suppose that $1/n \ll \mu \ll \gamma \leq 1$. Let $G = (A \cup B, E)$ be a balanced bipartite graph on $2n$ vertices with at least one perfect matching. Suppose that for every perfect matching $M$ of $G$ and for every $x = a_1b_1 \in M$ there are at least $\gamma n^2$ edges of $G$ that are $(x,M)$-switchable. Given a $\mu n$-bounded system of conflicts for $E(G)$, the probability that a uniformly random perfect matching of $G$ is conflict-free is at least $e^{-\mu^{1/2}n}$.
\end{lemma}
\begin{proof}
Let $\Omega=\Omega(G)$ be the set of perfect matchings of $G$ equipped with the uniform distribution. By assumption, note that $\Omega \neq \emptyset$.
Let $M \in \Omega$ be a perfect matching chosen uniformly at random. Let $\cF$ be a $\mu n$-bounded system of conflicts for $E(G)$. 

For each unordered pair of edges $x,y \in E(G)$ let $E(x,y) = \{x, y \in M \}$ be the event that both $x$ and $y$ are simultaneously in $M$. 
Define
$$
Q=\left\{\{x,y\}\in \cF:\,x,y \text{ non-incident}\right\}\;,
$$
and let $q=|Q|$.
Consider the collection of events $\cE=\{E(x,y): \{x,y\}\in Q\}$. 

Write $\cE = \{E_i : i \in [q] \}$ and let $\cH$ be the graph with vertex set $[q]$ where $i,j\in [q]$ are adjacent if and only if the subgraph of $G$ that contains the set of edges $\{x,y,w,z\}$ is not a matching, where $E_i=E(x,y)$ and $E_j=E(w,z)$. 

Observe that given $\{x,y\}\in Q$, there are at most $4n$ ways to choose an edge $w\in E(G)$ that is incident either to $x$ or to $y$, and at most $\mu n$ ways to choose an edge $z\in E(G)$ with $\{w,z\}\in \cF$. Hence, the maximum degree in $\cH$ is at most $d:=4 \mu n^2$. 

Our goal is to show that $\cH$ is a $p$-dependency graph for $\cE$, for a suitably small $p>0$. Given $i \in [q]$ and $S \subseteq [q] \sm N_\cH[i]$ with $\bP(\cap_{j \in S} {E_j}^c)>0$, it suffices to show that~\eqref{eq:dep_cond} holds.

Let $E_i = E(x,y)$. We say that a perfect matching is \emph{$S$-good} if it belongs to $\cap_{j \in S} E_j^c$. Since $\bP(\cap_{j \in S} {E_j}^c)>0$, there is at least one $S$-good perfect matching.
Let $\cM=\cM(S)$ be the set of $S$-good perfect matchings and let $\cM_0 \subseteq \cM$ the set of $S$-good perfect matchings that contain both $x$ and $y$. 

Construct an auxiliary bipartite graph $\cG = ( \cM_0, \cM \sm \cM_0, E(\cG))$, where $M_0 \in \cM_0$ and $M \in \cM$ are adjacent (i.e. $M_0M\in E(\cG)$) if there exist edges $x_1,x_2,y_1,y_2 \in M_0$ and $x_3,x_4,x_5, y_3,y_4,y_5 \in M$ such that $x, x_3, x_1, x_5, x_2, x_4$ and $y, y_3, y_1, y_5, y_2, y_4$ are vertex disjoint $6$-cycles contained in $G$ (see Figure~\ref{c6switch}).

By double-counting the edges of $\cG$, we obtain
\begin{equation*}
\delta (\cM_0) | \cM_0| \leq |E(\cG)| \leq \Delta(\cM \sm \cM_0) | \cM \sm \cM_0 |\;,
\end{equation*}
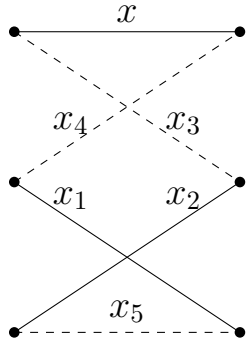
\begin{figure}[b]
\centering
\begin{tikzpicture}
\fill (0,1) circle(2pt);
\fill (0,3) circle(2pt);
\fill (0,5) circle(2pt);
\fill (3,1) circle(2pt);
\fill (3,3) circle(2pt);
\fill (3,5) circle(2pt);

\draw (0,5) -- (3,5) node[midway, above]{\Large $x$};
\draw (0,1) -- (3,3) node[pos=0.75, above]{\Large $x_2$};
\draw (0,3) -- (3,1) node[pos=0.25, above]{\Large $x_1$};

\draw [dashed] (0,1) -- (3,1) node[midway, above]{\Large $x_5$};
\draw [dashed] (0,3) -- (3,5) node[pos=0.25, above]{\Large $x_4$};
\draw [dashed] (0,5) -- (3,3) node[pos=0.75, above]{\Large $x_3$};
\end{tikzpicture}
\caption{Switching for edge $x$.\\ $x_5$ is $(x, M)$-switchable.}
\label{c6switch}
\end{figure}
from which we deduce,
\begin{equation}\label{eq:Dd}
\bP(E_i \vert \cap_{j \in S} E_j^c ) = \frac{|\cM_0|}{|\cM|} \leq \frac{|\cM_0|}{|\cM \sm \cM_0|} \leq \frac{\Delta( \cM \sm \cM_0)}{\delta(\cM_0)}\;.
\end{equation}
So, in order to prove~\eqref{eq:dep_cond} we need to bound $\Delta( \cM \sm \cM_0)$ from above and $\delta(\cM_0)$ from below.

We first bound  $\Delta( \cM \sm \cM_0)$ from above. Fix $M\in \cM\sm\cM_0$ and let us count the number of $6$-cycles of the form $xx_3x_1x_5x_2x_4$, with $x_3,x_4,x_5\in M$. Since $x_5 \in M$ is not incident with $x$, once we have chosen $x_5$ the $6$-cycle is completely determined, as the edges $x_3$ and $x_4$ are the ones in $M$ that are incident to both endpoints of $x$. There are at most $n$ choices for $x_5$, so there are at most $n$ $6$-cycles containing $x$. Similarly there are at most $n$ $6$-cycles containing $y$. It follows that $\Delta( \cM \sm \cM_0) \leq n^2$.

In order to bound $\delta(\cM_0)$ from below, fix $ M_0 \in \cM_0$. Note here that not all pairs of disjoint $6$-cycles containing $x$ and $y$, respectively, will generate an edge in $\cG$ as it may be that the perfect matching obtained by switching over the cycles is not $S$-good.

Define
$$
F(z)=\{z'\in E(G):\, z' \text{ is $(z, M_0)$-switchable and } \{ z, z' \} \not \in Q \}\;.
$$

Let $F^*(x)\subseteq F(x)$ be the subset of edges that are not incident with $y$. By assumption, there are at least $\gamma n^2$ edges that are $(x, M_0)$-switchable, from which at most $\mu n^2$ have conflicts with edges in $M_0$ and at most $2n$ are incident to $y$, implying $|F^*(x)|\geq \gamma n^2 /2$. Each edge $x_5\in F^*(x)$ uniquely determines a $6$-cycle whose switching gives rise to an $S$-good perfect matching. 
Given a choice of a $6$-cycle of the form $x x_3 x_1 x_5 x_2 x_4$, let $F^{**}(y)\subseteq F(y)$ be the subset of edges that are not incident to the vertices of the fixed $6$-cycle. Similarly as before,  $|F^{**}(y)|\geq \gamma n^2 /2$ and each edge $y_5\in F^{**}(y)$ gives rise to a $6$-cycle whose switching preserves the $S$-good condition.
We conclude that $\delta(\cM_0) \geq \gamma^2 n^4 / 4$.

Substituting into~\eqref{eq:Dd}, we obtain the desired bound
\begin{equation*}
\bP(E_i \vert \cap_{j \in S} E_j^c ) \leq \frac{4}{\gamma^2 n^2} =: p \;.
\end{equation*}
Note that $|\cE|\leq \mu n^3$.
Since, $4pd = \frac{64 \mu}{\gamma^2} \leq 1$, by the $p$-Lopsided form of the Lov\'{a}sz Local Lemma (Lemma~\ref{L4}), the probability that a uniformly random perfect matching is conflict-free is
$$
\bP(\cap_{E\in \cE} E^c)  \geq \left(1-\frac{8}{\gamma^2n^2}\right)^{\mu n^3} \geq e^{-\mu^{1/2} n}\;,
$$
and the lemma follows.
\end{proof}

\section{Dichotomy}\label{sec:dico}
In order to apply Lemma~\ref{key}, we  need to show that for every edge and every perfect matching containing it, there exist many switchable edges. However, this statement is not true for every Dirac bipartite graph. For instance, consider the graph where $n=2m+1$, $A=A_1\cup A_2$ and $B=B_1\cup B_2$ with $|A_1|=|B_2|=m+1$, and where $G[A_1,B_1]$ and $G[A_2,B_2]$ induce complete bipartite graphs and $G[A_2,B_1]$ induces a perfect matching. Clearly, $G$ is a Dirac bipartite graph. However, for every edge in $x\in E(A_1,B_2)$, and independently of the choice of $M$ containing $x$, there are at most $n$ edges that are $(x,M)$-switchable as any such edge lies in $E(A_1,B_2)$.

Our proof proceeds by splitting the class of Dirac bipartite graphs into two cases: Robust Expanders, where we show the existence of many switchable edges, and Extremal Graphs, where we proceed carefully to handle the edges that produce a small number of switchings.

For $0<\nu<1$ and $X\subseteq V(G)$, the \emph{$\nu$-robust neighbourhood} of $v$ in $G$ is defined as
$$
RN_{\nu}(X) \defeq \{ v \in V(G) : |N_G(v) \cap X| \geq \nu n \}\;.
$$
\begin{definition*}
Let $0<\nu\leq \tau<1$. A balanced bipartite graph $G = (A \cup B, E)$ on $2n$ vertices is a \emph{bipartite robust $(\nu, \tau)$-expander} if for every set $X\subseteq V(G)$ with $\tau n \leq |X| \leq (1-\tau)n$ and either $X\subseteq A$ or $X\subseteq B$, we have
\begin{equation*}
|RN_{\nu}(X)|  \geq |X| + \nu n\;.
\end{equation*}
\end{definition*}
\begin{definition*}
Let $0<\eps<1$. A balanced bipartite graph $G = (A \cup B, E)$ on $2n$ vertices is an \emph{$\eps$-extremal graph} if there exist partitions $A = \parts{A}{}$ and $B = \parts{B}{}$ such that the following properties are satisfied:
\begin{itemize}
\item[(P1)] $||A_1|-|A_2|| \leq \eps n$;
\item[(P2)] $||B_1|-|B_2|| \leq \eps n$;
\item[(P3)] $e(A_1,B_2) \leq \eps n^2$. 
\end{itemize}
\end{definition*}
The following result establishes a dichotomy between these classes. Similar ideas have already appeared in previous work on Dirac graphs~\cite{millie, KLScompatible, REepsKLOS}. The proof follows the lines of the previous approaches and we include it here for the sake of completeness.
\begin{theorem}
\label{dichotomy}
Let $n \in \bN$ and suppose that $1/n \ll \nu \ll \eps \ll \tau \ll 1$.
Let $G=(A \cup B, E)$ be a  Dirac bipartite graph on $2n$ vertices.
Then one of the following holds:
\begin{itemize}
\item[$i)$] $G$ is a bipartite robust $(\nu, \tau)$-expander;
\item[$ii)$] $G$ is an $\eps$-extremal graph.
\end{itemize}
\end{theorem}
\begin{proof}
Let $0<\delta<1$ be such that $\nu \ll \delta \ll \eps$.
Suppose that $G$ is not a bipartite robust $(\nu, \tau)$-expander. Thus, we may assume without loss of generality that there exists a set $X \subseteq A$ with $\tau n\leq |X|\leq (1-\tau)n$ and such that $|RN_{\nu}(X)| < |X|+\nu n$. We split the argument into three possible cases:
\begin{case}
$\tau n \leq |X| \leq \frac{n}{2} - \delta n$.
\end{case}
Since $e(X, N(X)) \geq |X| (n/2)$, we reach a contradiction:
$$e(X, N(X)) \leq |X| |RN_{\nu}(X)| + \nu n^2 \leq  |X| \left( \frac{n}{2} -(\delta-\nu)n\right) +\nu n^2 < |X| \frac{n}{2} \leq e(X, N(X))\;.
$$
\begin{case}
$\frac{n}{2} - \delta n \leq |X| \leq \frac{n}{2} + \delta n$.
\end{case}
Define $A_1 = X$, $A_2 = A \sm X$, $B_1 = RN_{\nu}(X)$, $B_2 = B \sm RN_{\nu}(X)$.
Note that $e(A_1, B_2) \leq \nu n |B_2| \leq \eps n^2$; thus, (P3) holds.
Now,  (P1) and (P2)  follow immediately since the minimum degree of $G$ is at least $n/2$.
Hence $G$ is an $\eps$-extremal graph.
\begin{case}
$\frac{n}{2} + \delta n \leq |X| \leq (1-\tau) n$.
\end{case}
Each vertex in $B$ has at least $\delta n$ neighbours in $X$.
So $RN_{\nu}(X) = B$.
Hence, $|RN_{\nu}(X)| =n \geq |X|+\nu n$, a contradiction with the choice of $X$.
\end{proof}

\section{Robust Expanders}\label{sec:robexp}

As we will show below, the robust expansion property yields to many $(x,M)$-switchable edges independently of the choice of $x$ and $M$. Thus, Lemma~\ref{key} can be directly applied to obtain the existence of a conflict-free perfect matching in robust expanders. 
\begin{lemma}
\label{expander}
Let $n\in \bN$ and suppose $1/n \ll \gamma \ll \nu \ll \tau \ll 1$.
Let $G = (A \cup B, E)$ be a bipartite robust $(\nu, \tau)$-expander on $2n$ vertices with minimum degree at least $n/2$.
Let $M$ be a perfect matching of $G$ and let $x \in M$.
Then, there are at least $\gamma n^2$ edges of $G$ that are $(x,M)$-switchable. 
\end{lemma}
\begin{proof}
Let $f_M:A\to B$ be a bijective map defined as $f(a)=b$ if and only if $ab\in M$. 
Given $x=a_1b_1 \in M$ and an edge $y \not \in M$ not incident to $x$, there is at most one $6$-cycle in $G$ that uses $x$, $y$ and any two of the other edges in $M$. 
Also, note $y$ is $(x, M)$-switchable if and only if there is such a $6$-cycle.
Therefore, to count the number of $(x,M)$-switchable edges $y$, we count the number of $6$-cycles containing $x$ and any two of the other edges in $M$. 

Suppose that the $6$-cycle is given by the sequence of vertices $a_1b_1a_2b_3a_3b_2$, where $a_2b_3,a_3b_2\in M$.
To bound from below the number of ways to choose the $6$-cycle, we compute a lower bound on the number of choices for $a_3$ and $b_3$.
Select,
\begin{align*}
a_3 & \in RN_{\nu} (f_M(N(b_1))) \cap f_M^{-1}(N(a_1)\sm \{b_1\})\;,
\end{align*}
and let $b_2 = f_M(a_3)$. Given the choice of $a_3$, select
\begin{align*}
b_3 & \in f_M(N(b_1) \sm \{a_1, a_3 \})\cap N(a_3) \;,
\end{align*}
and let $a_2 = f_M^{-1}(b_3)$. Recall that the minimum degree is at least $n/2$. As $G$ is a bipartite robust $(\nu,\tau)$-expander, $|RN_{\nu} (f_M(N(b_1)))|\geq \frac{n}{2}+\nu n$, which implies that there are at least $ \nu n -1$ choices for $a_3$. Again, by the expansion properties of $G$, $a_3$ has at least $\nu n$ neighbours in $f_M(N(b_1))$, so there are at least $\nu n -2$ choices for $b_3$.
In total, there are at least $\gamma n^2$ choices of $6$-cycles, $a_1b_1a_2b_3a_3b_2$, or equivalently,  $\gamma n^2$ edges $y=a_3b_3\in E(G)$ that are $(x,M)$-switchable.
\end{proof}
We can combine Lemma~\ref{key} and Lemma~\ref{expander} together to conclude.
\begin{corollary}
\label{RPMexpander}
Let $n\in \bN$ and suppose $1/n \ll \mu \ll \nu \ll \tau \ll 1$.
Let $G = (A \cup B, E)$ be a bipartite robust $(\nu, \tau)$-expander on $2n$ vertices with minimum degree at least $n/2$.
Then, any $\mu n$-bounded system of conflicts for $E(G)$ contains a conflict-free perfect matching.
\end{corollary}

\section{Extremal Graphs}\label{sec:ext}

In this section we study the existence of rainbow perfect matchings for extremal graphs.

The example displayed at the beginning of Section~\ref{sec:dico} suggests that extremal graphs have special edges that are difficult to switch; namely, the ones between $A_1$ and $B_2$. Since the partitions $A=(A_1,A_2)$ and $B=(B_1,B_2)$ can be unbalanced, it may be unavoidable to select edges in $E(A_1,B_2)$ in a perfect matching of $G$. In fact, we may have to choose linearly many such edges.

A greedy approach for choosing the edges in $E(A_1,B_2)$ is likely to fail. By the properties of the edge colouring, the graph may contain vertices that only have a constant number of colours in the edges incident to them. If one selects a partial matching $M^*$ in $E(A_1,B_2)$ and removes all the edges that have a colour in $M^*$, vertices that have few colours on their incident edges are likely to become isolated.

The way to handle this problem is given by Lemma~\ref{middle}, which shows that there is a way to select a rainbow partial matching $M^*$ in $E(A_1,B_2)$ such that $|A_i\sm V(M^*)|=|B_i \sm V(M^*)|$, for $i\in\{1,2\}$, and such that the degrees in the subgraph obtained after removing the colours in $M^*$ are similar to the ones in the original graph.

\subsection{A technical lemma}

The core of the proof of Lemma~\ref{middle} is a technical lemma that we present in this section.

We will be dealing both with multisets and with sets. We define the operators $\cap^+$ and $\sm^+$ both taking a multiset and a set and returning a multiset as follows: if $A$ is a multiset and $B$ is a set, 
\begin{align*}
A \cap^+ B \defeq \{ \{ x \in A : x \in B \} \} & & A \sm^+ B \defeq A \sm (A \cap^+B)
\end{align*}
where $\sm$ is the standard multiset difference.

%

\begin{lemma}
\label{boxes}
Let $\N\in \bN$ and suppose that $1/\N \ll \mu \ll  \nu, 1/\alpha \ll \eta \ll 1$.
Let $C_1, \ldots, C_\N$ be multisets of $\bN$ such that:
\begin{itemize}
\item[(B1)] $\nu \N \leq |C_i| \leq \N$, for every $i \in [\N]$;
\item[(B2)] $\sum_{i=1}^\N m(k, C_i) \leq \mu \N$, for every $i\in [\N]$ and every $k\in \bN$.
\end{itemize}
Let $\ell\in \bN$ with $1\leq \ell \ll \nu N$ and $\alpha\ell\in\bN$. Let $U \subseteq \bN$ be a set with $|U| = \alpha \ell$.
Then, there exists $T \subseteq U$ such that:
\begin{itemize}
\item[(T1)] $|T| \geq \ell$;
\item[(T2)] $|C_i \sm^+ T| \geq (1 - \eta) |C_i|$, for every $i \in [\N]$.
\end{itemize}
\end{lemma}
\begin{proof}
If $\ell \leq 2\alpha$, then let $T$ be an arbitrary subset of $U$ of size $\ell$. Since  for every $i\in [\N]$, we have $\mu \N \ell \leq \eta |C_i|$, $(T2)$ clearly holds. Throughout the proof we will assume that $\ell\geq 2\alpha$.

Let $0<\eps<1$ such that $\mu \ll \eps \ll \nu,1/\alpha$ and let $\cmax := \frac{\eps}{10 \alpha^2} \cdot \frac{\N}{\log \N}$ and let $s := \log (  \mu \N/\cmax )$.
For every $i \in [\N]$ and every $j \in \s$, define the (multi)sets
\begin{align*}
P_i^j & = \{ \{ k \in C_i : 2^{-j} \mu \N \leq m(k, C_i) \leq 2^{-(j-1)} \mu \N \} \} \\
S_i^j & = \{ k \in P_i^j \}
\end{align*}
Further, define
\begin{align*}
P_i & = \cup_{j \in \s} P_i^j \\
S_i & = \cup_{j \in \s} S_i^j \\
Q_i & = C_i \sm P_i
\end{align*}
Note that for every $k \in Q_i$, $m(k, C_i) \leq \cmax$.
Let $p_i^j = |P_i^j|$, $s_i^j = |S_i^j|$, $p_i = |P_i|$, $s_i = |S_i|$, $q_i = |Q_i|$, $c_i = |C_i|$. Then, these parameters satisfy
\begin{align}
2^{-j} \mu \N s_i^j & \leq p_i^j \leq 2^{-(j-1)} \mu \N s_i^j \label{sp}\\
\sum_{j \in \s} p_i^j & = p_i \label{sump} \\
p_i+q_i & = c_i \nonumber
\end{align}
Let $T_0 \subseteq U$ be a random subset of $U$ obtained by including each element of $U$ independently at random with probability $\delta := 3 \alpha^{-1}$.
Note that $\bE(|T_0|) = 3\ell$.
\begin{claim}
With probability $\whp{N}$, for every $i \in [\N]$, $|Q_i \sm^+ T_0| \geq (1 - 4 \alpha^{-1})q_i - \eps \N$.
\end{claim}
\begin{proof}
Fix $i\in [\N]$. If $q_i \leq \eps \N$ the statement is clearly true. So we may assume that $q_i \geq \eps \N$. For each $k \in Q_i$, define $m_k := m(k, C_i)$. 

Then,
\begin{equation}
\sum_{k \in Q_i} m_k^2 \leq \cmax \sum_{k \in Q_i} m_k = \cmax q_i \leq \frac{\eps \N}{10 \alpha^2 \log \N} \cdot q_i. \label{ck}
\end{equation}
Let $X_i = |Q_i \cap^+ T_0|$ and note that $\bE(X_i) \leq \delta q_i$. By Azuma's Inequality~(see e.g.~\cite{mrazuma}) with $m_k$ satisfying~\eqref{ck} and the fact that $q_i \geq \eps \N$,
\begin{equation*}
\bP(X_i - \bE (X_i) \geq \alpha^{-1} q_i) \leq 2 \exp \bigg( \frac{q_i^2}{2 \alpha^2 \sum_{k \in Q_i} m_k^2} \bigg) \leq \N^{-5}\;.
\end{equation*}
So, with probability $\whp{N}$, for every $i \in [\N]$, if $q_i \geq \eps \N$, then
\begin{equation*}
|Q_i \sm^+ T_0| \geq (1 - \alpha^{-1})q_i -\bE (X_i)= (1 - 4 \alpha^{-1}) q_i \geq (1 - 4 \alpha^{-1}) q_i- \eps \N\;.
\end{equation*}
\end{proof}
We now consider the sets $P_i$. 
For $\rho>0$, a pair $(i, j)$ is \emph{$\rho$-dense} if $s_i^j \geq 2^{(j-1)/2}\rho$. Let $R_i$ be the set of pairs $(i,j)$ that are $\mu^{-1/2}$-dense. 
The contribution of non-dense pairs is negligible; using~\eqref{sp}, we have
\begin{equation}
\sum_{j \not \in R_i} p_i^j \leq \mu \N \sum_{j \not \in R_i} 2^{-(j-1)}s_i^j \leq \mu^{1/2} \N \sum_{j \not \in R_i} 2^{-(j-1)/2} \leq \mu^{1/3}{\N}\;. \label{pij}
\end{equation}
We say that $i \in [\N]$ is \emph{susceptible} if $|C_i \cap^+ U| \geq \eta |C_i|$. Let $D = \{ i \in [\N] : i \text{ is susceptible} \}$.
Note that (T2) is satisfied for every $i\notin D$, as we have 
$$
|C_i \sm^+ T| \geq  |C_i \sm^+ U|= |C_i|-  |C_i \cap^+ U| \geq (1 - \eta) |C_i|\;.
$$
Since $|C_i|\geq \nu \N$, we can bound the size of $D$ as follows
\begin{equation}
|D| \leq \alpha \ell \cdot \mu \N / (\eta \nu \N) \leq  \ell \;. \label{susceptible}
\end{equation}
Finally, for every $S \subseteq \bN$ and $j\in [s]$ we say that $i \in [\N]$ is \emph{$j$-activated} by $S$ if $|S_i^j \cap S| \geq 2\delta s_i^j$.

Consider the set $T \subseteq T_0$ defined as follows: for each $i \in D$ and $j \in \s$ remove $S_i^j$ from $T_0$ if
\begin{itemize}
\item[i)] $i$ is $j$-activated by $T_0$, and
\item[ii)] $j\in R_i$ (i.e., $(i,j)$ is $\mu^{-1/2}$-dense).
\end{itemize}
Observe that by removing elements from $T_0$ we only increase the size of $Q_i \sm^+ T_0$.
From the construction of $T_0$ and using~\eqref{sp} twice, it follows that for each $i \in D$, $j \in R_i$, we have
\begin{equation*}
|P_i^j \cap^+ T| \leq \mu \N 2^{-(j-1)}|S_i^j \cap T| \leq \mu \N 2^{-(j-1)} \cdot 2 \delta s_i^j  \leq 4\delta p_i^j 
\end{equation*}
By combining this with~\eqref{pij}, we obtain
\begin{equation*}
|P_i \cap^+ T| = \sum_{j \in \s} |P_i^j \cap^+ T| = \sum_{j \in R_i} |P_i^j \cap^+ T| + \sum_{j \not \in R_i} |P_i^j \cap^+ T| \leq 4 \delta \sum_{j \in R_i} p_i^j + \sum_{j \not \in R_i} p_i^j \leq 4 \delta p_i + \mu^{1/3} \N.
\end{equation*}
Therefore, with probability $\whp{N}$, condition $(T2)$ is satisfied; that is, for every $i \in [\N]$,
\begin{align*}
|C_i \sm^+ T| &= |P_i \sm^+ T| + |Q_i \sm^+ T| \\
&\geq |P_i \sm^+ T| + |Q_i \sm^+ T_0| \\
&\geq (1-4 \delta)p_i -\mu^{1/3}\N + (1 - 4 \alpha^{-1})q_i - \eps \N \geq (1-\eta) |C_i| \label{eta}
\end{align*}
In order to conclude the proof of the lemma, it suffices to show that condition $(T1)$ holds with positive probability, from where we will deduce the existence of the desired set.
\begin{claim}
With probability at least $\frac{9}{10}$, we have $|T| \geq |T_0|-\ell$.
\end{claim}
\begin{proof}
Since $|S_i^j \cap T_0|$ is stochastically dominated by a binomial random variable with parameters $s_i^j$ and $\delta$ (there might be elements of $S_i^j$ that are not in $U$), we can use Chernoff's inequality~(see e.g. Corollary 2.3 in~\cite{JLRchernoff}) to show that
\begin{equation*}
\bP (i \text{ is } j \text{-activated}) \leq 2e^{-\frac{\delta s_i^j}{3}}\;.
\end{equation*}
If $j \in R_i$, then $s_i^j \geq 2^{(j-1)/2} \mu^{-1/2}$. Thus, $e^{-\frac{\delta s_i^j}{3}} \leq e^{-\frac{\delta 2^{(j-1)/2}}{3 \mu^{1/2}}} \leq \mu^2 2^{-j}$.
Hence, for $j \in R_i$
\begin{equation} 
\bP (i \text{ is } j \text{-activated}) \leq \mu^2 2^{-j} \;.\label{activated}
\end{equation}
Recall the following inequality which follows from~\eqref{sp} and~\eqref{sump},
\begin{equation}
\sum_{j \in \s} 2^{-j}s_i^j \leq \mu^{-1} \;.\label{sij}
\end{equation}
Define the following random variable
\begin{equation*}
Y \defeq |T_0 \sm T| 
\leq \sum_{i\in D}\sum_{j\in R_i} s_i^j\mathbbm{1} (i \text{ is } j\text{-activated})\;.
\end{equation*}
Note that the sets $D$ and $R_i$ are fully determined by $C_1,\dots, C_\N$.
Then using~\eqref{susceptible},~\eqref{activated} and~\eqref{sij}, it follows that
\begin{equation*}
\bE (Y) \leq \sum_{i\in D}\sum_{j\in R_i} s_i^j \bP (i \text{ is } j\text{-activated}) \leq \mu^2 \sum_{i\in D}\sum_{j\in R_i} 2^{-j}s_i^j  \leq \mu^2 \sum_{i\in D}\sum_{j\in \s} 2^{-j} s_i^j  \leq \mu |D| \leq \frac{\ell}{10}\;.
\end{equation*}
So, by Markov's inequality, $\bP(Y \geq \ell) \leq 1/10$.
\end{proof}
Recall that $\ell\geq 2\alpha$. Since $|T_0|$ is distributed as a binomial random variable with parameters $\alpha \ell$ and $\delta$, Chernoff's inequality implies that $\bP(|T_0| \leq 2\ell) \leq 2e^{-\frac{\ell^2}{2 \alpha \ell}} = 2 e^{-\frac{\ell}{2\alpha}} \leq \frac{2}{e}$.
Thus, with positive probability, we have
\begin{equation}
|T| \geq 2\ell-\ell \geq \ell \label{bigT}
\end{equation}
We conclude that there exists $T \subseteq U$ satisfying~(T1) and~(T2), concluding the proof of the lemma.
\end{proof}

\subsection{Superextremal graphs}

We will use Lemma~\ref{boxes} to control the effect of colour deletions in the degrees of $G$. If degrees do not shrink significantly, the graphs $G_i=G[A_i,B_i]$, $i\in\{1,2\}$, will still be fairly dense, and by applying Lemma~\ref{key} we will get the existence of a rainbow perfect matching.

However, the $\eps$-extremal condition does not ensure that the graphs $G_i$ have large minimum degree; that is, $G_i$ is not necessarily Dirac. In this section we refine the notion of extremality and we obtain a partition where the degrees of each vertex within its part is controlled. Eventually, this will allow us to count the number of switchable edges.

\begin{definition*}
Let $0< \nu_1 \leq \nu_2<1$. A balanced bipartite graph $G = (A \cup B, E)$ on $2n$ vertices is a \emph{$(\nu_1,\nu_2)$-superextremal graph} if there exist partitions $A = \parts{A}{}$ and $B = \parts{B}{}$ such that the following properties are satisfied for $i\in\{1,2\}$:
\begin{itemize}
\item[(Q1)] $e(v, B_i) \geq \frac{n}{2} - \nu_1 n$, for all but at most $\nu_1 n$ vertices $v \in A_i$;
\item[(Q2)] $e(v, B_i) \geq \nu_2 n$, for every $v \in A_i$;
\item[(Q3)] $e(v, A_i) \geq \frac{n}{2} - \nu_1 n$, for all but at most $\nu_1 n$ vertices $v \in B_i$;
\item[(Q4)] $e(v, A_i) \geq \nu_2 n$, for every $v \in B_i$;
\item[(Q5)] $| \vert A_1 \vert - \vert B_1 \vert |$, $| \vert A_1 \vert - \vert A_2 \vert | \leq \nu_1 n$;
\item[(Q6)] $e(v, B_2) \leq \nu_2 n$, for every $v \in A_1$, unless $|A_1|=|B_1|$;
\item[(Q7)] $e(v, A_1) \leq \nu_2 n$, for every $v \in B_2$, unless $|A_1|=|B_1|$;
\item[(Q8)] $|A_1|\geq |B_1|$;
\item[(Q9)] one of the following holds for $\ell := |A_1|-|B_1|$:
\begin{itemize}
\item $e(v, B_2) \geq \ell/2$, for every $v \in A_1$;
\item $e(v, A_1) \geq \ell/2$, for every $v \in B_2$.
\end{itemize}
\end{itemize}
\end{definition*}
\begin{lemma}
Let $n\in \bN$ and suppose $1/n \ll \eps \ll \nu_1 \ll \nu_2 \ll  1$.
Let $G = (A \cup B, E)$ be an $\eps$-extremal Dirac bipartite graph on $2n$ vertices.
Then, $G$ is a $(\nu_1,\nu_2)$-superextremal graph.
\end{lemma}


\begin{proof}
Since $G$ is an $\eps$-extremal graph, there exist partitions $A = \parts{A}{1}$ and $B = \parts{B}{1}$ satisfying $(P1)$, $(P2)$ and $(P3)$.
Let $0<\nu_3\leq \nu_4<1$ such that $\eps \ll \nu_3 \ll \nu_1 \ll \nu_4 \ll \nu_2$ and define
\begin{align*}
X_1^1 = \left\{ v \in A_1^1 : e(v, B_1^1) \leq \frac{n}{2} - \nu_3 n \right\} & & X_1^2 = \left\{ v \in A_1^1 : e(v, B_1^1) \leq \frac{n}{4} \right\} \\
X_2^1 = \left\{ v \in A_2^1 : e(v, B_2^1) \leq \frac{n}{2} - \nu_3 n \right\} & & X_2^2 = \left\{ v \in A_2^1 : e(v, B_2^1) \leq \frac{n}{4} \right\} \\
Y_1^1 = \left\{ v \in B_1^1 : e(v, A_1^1) \leq \frac{n}{2} - \nu_3 n \right\} & & Y_1^2 = \left\{ v \in B_1^1 : e(v, A_1^1) \leq \frac{n}{4} \right\} \\
Y_2^1 = \left\{ v \in B_2^1 : e(v, A_2^1) \leq \frac{n}{2} - \nu_3 n \right\} & & Y_2^2 = \left\{ v \in B_2^1 : e(v, A_2^1) \leq \frac{n}{4} \right\}
\end{align*}
We double count edges to bound the size of these sets.
Note that $e(A_1^1, B_1^1) \geq \frac{n}{2} |A_1^1| - \eps n^2$ by counting from $A_1^1$.
Alternately, we can also obtain that $e(A_1^1, B_1^1) \leq |X_1^1|(\frac{n}{2}-\nu_3 n)+(|A_1^1|-|X_1^1|)|B_1^1|$.
Combining these two inequalities yields
\begin{equation*}
|X_1^1|\left(|B_1^1| - \frac{n}{2} + \nu_3 n\right) \leq |A_1^1|\left(|B_1^1| - \frac{n}{2}\right) + \eps n^2 \leq 2 \eps n^2.
\end{equation*}
Observe that $|B_1^1| \geq \frac{n}{2}-\eps n$ and so $|B_1^1| - \frac{n}{2} + \nu_3 n  \geq \frac{\nu_3 n}{2}$. Therefore $|X_1^1| \leq \nu_3 n$.
Similarly, one can deduce that $|X_1^2| \leq 9 \eps n$. Analogous computations lead to $|X_2^1|, |Y_1^1|, |Y_2^1|\leq \nu_3 n$ and to $|X_2^2|, |Y_1^2|, |Y_2^2|\leq 9\eps n$.
Now, we define
\begin{align*}
A_1^2 = (A_1^1 \sm X_1^2) \cup X_2^2 & & B_1^2 = (B_1^1 \sm Y_1^2) \cup Y_2^2 \\
A_2^2 = (A_2^1 \sm X_2^2) \cup X_1^2 & & B_2^2 = (B_2^1 \sm Y_2^2) \cup Y_1^2
\end{align*}
Without loss of generality, $|A_1^2| \geq |B_1^2|$; otherwise we swap the labels of $A_1^2$ and $A_2^2$, and the labels of  $B_1^2$ and $B_2^2$. By swapping the labels we lose control of $e(A^2_1,B^2_2)$. However, at this point, this condition is no longer needed, as we have a bound on the size of the sets $X_i^j$ and $Y_i^j$, for $i,j\in\{1,2\}$. 

Let
\begin{align*}
X_1^3 = \{ v \in A_1^2 : e(v, B_2^2) \geq \nu_4 n \} & & Y_2^3 = \{ v \in B_2^2 : e(v, A_1^2) \geq \nu_4 n \}
\end{align*}
If $|X_1^3| + |Y_2^3| \geq |A_1^2|-|B_1^2|$, choose $X_1^4 \subseteq X_1^3$ and $Y_2^4 \subseteq Y_2^3$ arbitrarily such that $|X_1^4| + |Y_2^4| = |A_1^2|-|B_1^2|$.
Otherwise, let $X_1^4 = X_1^3$ and $Y_2^4 = Y_2^3$. Recall that, since $G$ is $\eps$-extremal, it satisfies $n/2-\eps n \leq |A_1^1|,|B_1^1|\leq n/2+\eps n$. Thus, we have 
$$
|X_1^4|\leq |A_1^2|-|B_1^2| \leq |A_1^1|-|B_1^1|+18\eps n\leq 20\eps n\;.
$$
and similarly for $Y_2^4$.

We define
\begin{align*}
A_1 = A_1^2 \sm X_1^4 & & A_2 = A_2^2 \cup X_1^4 & & B_1 = B_1^2 \cup Y_2^4 & & B_2 = B_2^2 \sm Y_2^4
\end{align*}
We claim that the partitions $A = \parts{A}{}$ and $B = \parts{B}{}$ satisfy properties (Q1)-(Q9), and so, $G$ is a $(\nu_1,\nu_2)$-superextremal graph. 

Let us first check that property (Q1) is satisfied. 
Observe that all the vertices in $A_1^1$, excluding the ones in $X_1^1$, have degree at least $n/2 - \nu_3 n$. Since $|X_1^2|\leq 9\eps n$, all the vertices in $A_1^2$ have degree at least $n/2 - \nu_3 n-9\eps n$, excluding the ones in $X_1^1\cup X_2^2$. Since $|X_1^4|\leq 20\eps n$, all the vertices in $A_1$ have degree at least $n/2 - \nu_3 n-29\eps n\geq n/2 -\nu_1 n$, excluding the ones in $X_1^1\cup X_2^2$. Moreover, $|X_1^1\cup X_2^2|\leq \nu_3 n+9\eps n\leq \nu_1 n$, so (Q1) follows. Similar arguments yield to properties (Q2)-(Q4) and (Q6)-(Q7).

Property (Q8) follows from the choice of $X_1^4$ and $Y_2^4$,  since $|A_1| = |A_1^2|-|X_1^4| \geq |B^2_1| + |Y_2^4|= |B_1|$. Property (Q5) follows since 
$|A_1|-|B_1| \leq |A_1^2|-|B_1^2| \leq 20\eps n \leq \nu_1 n$ (and since a similar computation bounds $||A_1|-|A_2||$).

Finally, Property (Q9) follows by noting that if $\ell=|A_1| - |B_1|$ (and $|A_2| = |B_2| - \ell$), then either $|B_1|$ or $ |A_2|$ is at most $n/2 - \ell/2$ thus requiring minimum degree $\ell/2$ either from $A_1$ to $B_2$ or from $B_2$ to $A_1$, respectively.

\end{proof}

\subsection{Selecting a rainbow partial matching between parts}
Given a superextremal graph $G$ with partitions $A=\parts{A}{}$ and $B=\parts{B}{}$, in this section we will show the existence of a  rainbow partial matching $M^*$ in $G[A_1,B_2]$ of size $\ell=|A_1|-|B_1|$ such that the graph $H$ resulting from removing all edges incident to $M^*$ and all edges with colours that appear in $M^*$, has similar degrees as the graph $G$.

\begin{lemma}
\label{middle}
Let $n,\ell\in \bN$ and suppose $1/n \ll \mu \ll \nu_1 \ll \nu_2 \ll \nu_3\ll \eta_1 \ll 1$.
Let $G = (A \cup B, E)$ be a $(\nu_1, \nu_3)$-superextremal graph with partitions $A = \parts{A}{}$ and $B= \parts{B}{}$.
Then, any $\mu n$-bounded edge colouring $\chi$ of $G$ admits a rainbow matching $M^*$ of size $\ell = |A_1|-|B_1|$ such that the following holds. Let $H=(A^H\cup B^H,E^H)$ be the graph where $A^H=A\sm V(M^*)$, $B^H=B\sm V(M^*)$ and 
$$
E^H=\{x=ab\in E(G):\, a,b\notin V(M^*),\,\chi(x)\notin \chi(E(M^*))\}\;.
$$
Let $n_H:=n-\ell$. Then, there exist partitions $A^H=A_1^H\cup A_2^H$ and $B^H = B_1^H\cup B_2^H$ that satisfy the following properties for $i\in\{1,2\}$:
\begin{itemize}
\item[(R1)] $e_H(v, B_i^H) \geq (1- \eta_1) \frac{n_H}{2}$, for all but at most $\nu_1 n$ vertices $v \in A_i^H$;
\item[(R2)] $e_H(v, B_i^H) \geq  \nu_2 n_H$, for every $v \in A_i^H$;
\item[(R3)] $e_H(v, A_i^H) \geq (1- \eta_1) \frac{n_H}{2}$, for all but at most $\nu_1 n$ vertices $v \in B_i^H$;
\item[(R4)] $e_H(v, A_i^H) \geq \nu_2 n_H$, for every $v \in B_i^H$;
\item[(R5)] $\vert A_1^H \vert = \vert B_1^H \vert$, $\vert A_2^H \vert = \vert B_2^H \vert$ and $\vert A_1^H \vert - \vert A_2^H \vert \leq \nu_1 n_H$.
\end{itemize}
\end{lemma}
\begin{proof}
We first greedily select a large rainbow matching in $G[A_1,B_2]$.
Let $E_0 = E(A_1, B_2)$ and $M_0 = \emptyset$.
By (Q5) and (Q8), note that $|E_0| \geq \frac{\ell}{2}(\frac{n}{2} - \nu_1 n)$.
For every $i\geq 1$ and while $E_{i-1} \neq \emptyset$, we arbitrarily choose $x_i=a_ib_i \in E_{i-1}$ and define the graph $M_i$ with $V(M_i)=V(M_{i-1})\cup\{a_i,b_i\}$ and $E(M_i) = E(M_{i-1}) \cup \{ x_i \}$.
We let 
$$
E_i = \{x=ab\in E_{i-1}:\, a,b\notin V(M_i),\,\chi(x)\notin \chi(E(M_i)) \}\;.
$$ 
Since $\chi$ is $\mu n$-bounded, $|A_1|-|B_1|=\ell\geq 1$ and using (Q6)-(Q7), we have  $|E_i| \geq |E_{i-1}| - (2\nu_3+\mu)n$. Let $i^*=\lfloor \ell/(10 \nu_3)\rfloor$. It follows that $E_i \neq \emptyset$,  for every $0\leq i\leq i^*$. 

We now apply Lemma \ref{boxes} with parameters $N=2n$, $\alpha = i^*/\ell$, $\nu = \nu_3 / 2$, $\eta = \eta_1/2$ and $U = \{ \chi(x) : x \in M_{i^*} \}$.
For every $v \in A \cup B$, we choose $C_v = \{ \{ \chi(x) : v \in x \} \}$ to be the multiset of colours on edges incident with vertex $v$.
By (Q2) and (Q4), we have $\nu N \leq |C_v| \leq N$, for each $v \in A \cup B$.
As each edge has two endpoints and $\chi$ is $\mu n$-bounded, then $\sum_{v \in A \cup B} m(k, C_v) \leq 2 \mu n = \mu N$.
Hence, (B1) and (B2) hold.\\
Lemma \ref{boxes} implies the existence of a set of colours $T \subseteq U$ of size $\ell$ satisfying (T1) and (T2).
Let $M^*$ be the subgraph of $M_{i^*}$ induced by the colours in $T$. Then, $M^*$ is a rainbow matching of size $\ell$. It suffices to prove that $H$, as defined in the statement, satisfies (R1)-(R5).

For each $Z \in \{ A,B \}$ and $i \in \{1,2 \}$,  let $Z_i^H=Z_i\cap V(H)$.
Property (R5) follows since $|B_1^H|=|B_1|= |A_1|-\ell = |A_1^H|$ and using (Q5).
 Then, for every $v\in Z_i^H$, we have
\begin{multline*}
e_H(v, Z_i^H) \geq |C_v \sm^+ T| - \ell \geq (1 - \eta) |C_v| - \nu_1 n \geq \\
\begin{cases}
(1 - \eta)(\frac{n}{2} - \nu_1 n) - \nu_1 n  \geq (1- \eta_1) \frac{n_H}{2} & \text{ if $v$ satisfies (Q1) or (Q3) } \\
(1 - \eta)\nu_3 n -\nu_1 n  \geq  \nu_2 n_H & \text{ if $v$ satisfies (Q2) or (Q4) }
\end{cases}
\end{multline*}
Thus $H$ satisfies (R1)-(R4), completing the proof.
\end{proof}

\subsection{Completing  the rainbow perfect matching}

Consider the  rainbow  partial matching $M^*$ and the graph $H$ provided by Lemma~\ref{middle}. Note that $H$ is vertex disjoint from $M^*$ and has no edge with colour in $\chi(E(M^*))$. Thus, the union of any rainbow perfect matching of $H$ and $M^*$ will provide a rainbow perfect matching of $G$.

We will show that $H$ satisfies the conditions of Lemma~\ref{key}, to conclude the existence of a rainbow perfect matching there. 

Of course, in order to have a rainbow perfect matching in $H$ we need to ensure the existence of at least one perfect matching. We will use the Moon-Moser condition for the existence of Hamiltonian cycles in bipartite graphs to guarantee we can find a perfect matching. 
\begin{lemma}
\label{mmthm}{(Moon, Moser \cite{mmHC})}
Let $F=(R \cup S, E)$ be a balanced bipartite graph on $2m$ vertices with $R=\{r_1,\dots,r_m\}$ and $S=\{s_1,\dots, s_m\}$ that satisfies $d(r_1) \leq \ldots \leq d(r_m)$ and  $d(s_1) \leq \ldots \leq d(s_m)$. Suppose that for every $1\leq k \leq m/2$, we have $d(r_{k}) > k $ and $d(s_{k}) > k $.
Then $F$ has a Hamiltonian cycle.
\end{lemma}

\begin{lemma}\label{lem:dense}
Let $n_H\in\bN$ and suppose $1/n_H \ll \mu \ll \eps \ll \nu_1\ll \gamma \ll \nu_2 \ll \eta \ll 1$.
Let $H=(A^H\cup B^H, E^H)$ be a bipartite graph with $|A^H|=|B^H|=n_H$ and satisfying properties (R1)-(R5). 
Consider the subgraph $H_*=(A^H\cup B^H,E^H_*)$ of $H$ with $E^H_*=E^H_{*,1}\cup E^H_{*,2}$ and, for $i\in\{1,2\}$,
$$
E^H_{*,i}=\left\{x=ab\in E(H):\, a\in A_i,\,b\in B_i \text{ and } \max\{e_H(a,B_i), e_H(b,A_i)\}\geq (1-\eta)n_H/2\right\}\;.
$$
Then, $H_*$ has at least one perfect matching.

Moreover, if $M_*$ is a perfect matching of $H_*$ and $x=a_1b_1 \in E(M_*)$, then there are at least $\gamma n_H^2$ edges of $H_*$ that are $(x,M_*)$-switchable.
\end{lemma}
\begin{proof}
It is easy to check that $H_*$ has only two connected components. 
We will show that it is true in $H_1=H_*[A^H_1,B^H_1]$ and the same argument also applies to $H_*[A^H_2,B^H_2]$.
Note that $H_1$ is a balanced bipartite graph on $2m$ vertices for some $m \in (n_H (1/2 - \nu_1), n_H (1/2 +\nu_1))$.

 We will use the Moon-Moser condition (Lemma~\ref{mmthm}) to show the existence of a Hamiltonian cycle in $H_1$.  Let $A^H_1=\{r_1,\dots, r_{m}\}$ and $B^H_1=\{s_1,\dots, s_{m}\}$ with $d(r_1)\leq \dots \leq d(r_m)$ and $d(s_1)\leq \dots \leq d(s_m)$.

If $1\leq k\leq 2\nu_1 n_H < 5 \nu_1 m -1$, then, by (R2), $d(r_k) \geq \nu_2 n_H \geq 5\nu_1 m > k$, so there is nothing to prove. If $2\nu_1 n_H \leq k \leq m/2$, then, by (R1), $d(r_k)\geq (1-\eta)m > k$. An identical argument works for $s_k$ using (R3) and (R4). Thus we satisfy the Moon-Moser condition.
So, $H_1$ has a Hamiltonian cycle, which implies the existence of a perfect matching.

Let $M_*$ be a perfect matching of $H_*$. Consider the bijective map $f_{M^*}:A^H\to B^H$ defined as $f(a)=b$ if and only if $ab\in M_*$.
Let $x = a_1b_1 \in M_*$, and, without loss of generality, assume that $a_1 \in A_1^H$, so $b_1\in B_1^H$. In order to prove the second part of the lemma, we need to show that there are many edges $y=ab$ that are $(x, M_*)$-switchable.

Let $0<\delta<1$ such that $\gamma\ll\delta\ll \nu_2$. Observe that the minimum degree in $H_*$ is at least $(\nu_2-\nu_1)n_H\geq \delta m$.
By construction, there is no pair of vertices both of degree less than $(1- \eta)m$ that are connected by an edge in $H_*$.
Thus, without loss of generality, we may assume that $e_{H_*}(a_1,B_1^H) \geq \delta m$ and that $e_{H_*}(b_1,A_1^H) \geq (1- \eta)m$.



Since $|f_{M_*}^{-1}(N_{H_*}(a_1))| \geq \delta m$  and since there are at most 
$2\nu_1 m$ vertices of degree less than $(1-\eta)m$, there are at least 
$\delta m/2$ choices for $a\in f_{M_*}^{-1}(N_{H_*}(a_1)\sm \{b_1\})$ that satisfies $e_{H_*}(a, B_1^H)\geq (1-\eta)m$. 

Fix such a vertex $a$ and note that
$$
e_{H_*}\left(a,B_1^H \sm f_{M_*}(N_{H_*}(b_1)\sm\{a_1,a\})\right) \leq |B_1^H|-(1- \eta)m+2 \leq  \eta m+2\;.
$$
Therefore, 
\begin{align*}
e_{H_*}\left(a,f_{M_*}(N_{H_*}(b_1)\sm \{a_1,a\})\right) 
&= e_{H_*}(a,B_1^H) - e_{H_*}\left(a,B_1^H \sm f_{M_*}(N_{H_*}(b_1)\sm\{a_1,a\})\right) \\
&\geq  (1-\eta)m -  (\eta m+2) \\
&\geq (1-3\eta) m\;.
\end{align*}
Thus, there are at least $(1-3\eta)m$ choices for $b\in f_{M_*}(N_{H_*}(b_1)\sm \{a_1,a\})$ with $ab\in E(H_*)$. It follows that there are at least $(\delta m/2)(1-3\eta)m\geq \gamma n_H^2$ choices of an edge $y=ab\in E(H_*)$ such that there exists a $6$-cycle that contains $x$, $y$ and two other edges of $M_*$. We conclude that there are at least $\gamma n_H^2$ edges of $H_*$ that are $(x, M_*)$-switchable. 

\end{proof}
The following corollary follows directly from the application of Lemma~\ref{key}, Lemma~\ref{middle} and Lemma~\ref{lem:dense}.
\begin{corollary}
\label{RPMextremal}
Let $n\in \bN$ and suppose $1/n \ll \mu \ll \eps \ll 1$.
Let $G = (A \cup B, E)$ be an $\eps$-extremal Dirac bipartite graph on $2n$ vertices.
Then, any $\mu n$-bounded edge colouring of $G$ contains a rainbow perfect matching.
\end{corollary}

\section{Proofs of Theorem~\ref{mainthm} and Theorem~\ref{conflicts}}\label{sec:final}
We finally prove our main theorems.
\begin{proof}[Proof of Theorem~\ref{mainthm}]
Let $G$ be a  Dirac bipartite graph on $2n$ vertices and suppose $1/n \ll \mu \ll \eps \ll \nu \ll \tau \ll 1$.
Consider a $\mu n$-bounded edge colouring $\chi$ of $G$.
By Lemma~\ref{dichotomy}, the graph $G$ is either $\eps$-extremal or a bipartite robust $(\nu, \tau)$-expander.
If $G$ is a bipartite robust $(\nu, \tau)$-expander, then $G$ has a rainbow perfect matching by Corollary~\ref{RPMexpander} with $\cF=\cF_\chi$.
If $G$ is an $\eps$-extremal graph, then $G$ has a rainbow perfect matching by Corollary~\ref{RPMextremal}.
\end{proof}

\begin{proof}[Proof of Theorem~\ref{conflicts}]
Let $G = (A \cup B, E)$ be a balanced bipartite graph on $2n$ vertices with minimum degree at least $(1/2 + \eps) n$. 
Suppose that  $1/n \ll \mu \ll \eps  \ll 1$. We will show that $G$ is a bipartite robust $(\eps/8, 1/4)$-expander. Let $X$ be a subset of either $A$ or $B$ with $n/4\leq |X|\leq 3n/4$; without loss of generality, we may assume that $X \subseteq A$.
By the minimum degree condition we have $e(X, B) \geq  (1/2 + \eps) n |X|$ and, by the definition of robust neighbourhood, we have $e(X, B) \leq |X| |RN_{\eps/8}(X)| + \eps n (n - |RN_{\eps/8}(X)|)/8$.
Combining these inequalities yields $|X||RN_{\eps/8}(X)| + \eps n^2/8 \geq (1/2 + \eps) n |X|$ and, as $|X| \geq n/4$, upon rearrangement, we have that $|RN_{\eps/8}(X)| \geq (1/2 + \eps/2)n$. 
If $n/4\leq |X| \leq n/2$, then $|RN_{\eps/8}(X)| \geq |X| + \eps n/8$ and we are done. If  $n/2 \leq |X| \leq 3n/4$, by the minimum degree condition, each $v \in B$ has at least $\eps n$ neighbours in $X$.
Thus $RN_{\eps/8}(X) = B$ and  $|RN_{\eps/8}(X)| = n \geq |X| + \eps n/8$. So $G$ is a bipartite robust $(\eps/8, 1/4)$-expander. Corollary~\ref{RPMexpander} completes the proof.
\end{proof}

The following proposition shows that $\mu\leq 1/4$~(see Section~\ref{sec:fur_rem} for a discussion).
\begin{proposition}
For every $t\in \mathbb{N}$, there exists a Dirac bipartite graph $G$ on $n=4t(t+1)$ vertices and a $\left(\frac{t+1}{4t}\,n\right)$-bounded edge colouring of $G$ such that $G$ does not contain a rainbow perfect matching.
\end{proposition}
\begin{proof}
Let $m=2t$. Consider the bipartite graph $G=(A\cup B,E)$ constructed as follows. The vertex set is partitioned into  $A=A_1\cup A_2$ and $B=B_1\cup B_2$, with 
\begin{align*}
A_1&=\{A_1^1,\dots, A_1^{m-1}\}\\
A_2&=\{A_2^1,\dots, A_2^{m+1}\}\\
B_1&=\{B_1^1,\dots, B_1^{m+1}\}\\
B_2&=\{B_2^1,\dots, B_2^{m-1}\}\;,
\end{align*}
where $|A_k^i|=|B_k^i|=t+1$. The edge set of $G$ is consists of two complete bipartite graphs induced by $G[A_1,B_1]$ and $G[A_2,B_2]$, and of $m+1$ smaller complete bipartite graphs induced by $G[A_2^i,B_1^i]$, for $i\in [m+1]$. Note that $G$ is a Dirac bipartite graph.

Consider the edge colouring that assigns colour $c_{k,\ell}^{i,j}$ to the edges in $G[A_k^i,B_\ell^j]$. Since each set has size $t+1$, the colouring is $(t+1)^2=\left(\frac{t+1}{4t}\,n\right)$-bounded.

Suppose that $G$ admits a rainbow perfect matching $M$. Note that $M$ contains at most $m+1$ edges in $G[A_2,B_1]$. Otherwise there exists $i\in[m+1]$ such that $M$ contains two edges in $E[A_2^i,B_1^i]$, contradicting the fact that it is rainbow, since both edges have colour $c_{2,1}^{i,i}$. Since all the edges incident to $A_1$ are also incident to $B_1$, we must have $|A_1|\geq |B_1|-(m+1)$. However
$$
|A_1| = (m-1)(t+1) = (m+1)(t+1)-2(t+1)=|B_1|-(m+2),
$$
a contradiction. We conclude that, $G$ has no rainbow perfect matching.

\end{proof}

\section{Applications}\label{sec:appl}
In the following section we provide some applications of our main theorems on the existence of rainbow spanning subgraphs of graphs with large minimum degree that are not necessarily bipartite. 

We first discuss the existence of rainbow $\Delta$-factors in Dirac graphs for a wide range of $\Delta$.  Recall that a \emph{Dirac graph} on $n$ vertices is a graph with minimum degree at least $n/2$.
The existence of $(n/2)$-factors in Dirac graphs was proved by Katerinis~\cite{kfactorexists}. 
Our next result extends Theorem~\ref{mainthm} to $\Delta$-factors of Dirac graphs.

%
%

\begin{theorem}
\label{rainbowkfactor}
There exist $\mu > 0$ and $n_0 \in \bN$ such that if $n \geq n_0$ and for every even $1 \leq \Delta \leq \mu n$ the following holds. Let $G$ be a Dirac graph on $n$ vertices, then any $(\mu n/\Delta)$-bounded colouring of $G$ contains a rainbow $\Delta$-factor.
\end{theorem}

Note that this theorem is tight in its dependence on $n$ and $\Delta$ as a $\Delta$ factor contains $n \Delta/2$ edges.
\begin{proof}
We construct an auxiliary bipartite graph $Q=(V(Q),E(Q))$ as follows.
The vertex set is $V(Q) = A\cup B$, where $A= \{ u_{v,i} :v \in V(G), 1\leq i \leq \Delta/2 \}$ and $B= \{ u_{v,i} :v \in V(G), \Delta/2< i \leq \Delta \}$.
The edge set is defined as 
$$
E(Q) = \{ u_{v,i} u_{w,j} :\, u_{v,i}\in A,\, u_{w,j}\in B \text{ and } vw \in E(G) \}\;.
$$
Note that $Q$ is a bipartite Dirac graph on $2N=\Delta n$ vertices. Let $\chi:E(G)\to \mathbb{N}$ be a $\mu n$-bounded edge colouring of $G$. Construct the edge colouring $\chi_Q: E(Q)\to \mathbb{N}$ defined by $\chi_Q(u_{v,i} u_{w,j})=\chi(vw)$, for every $u_{v,i} u_{w,j}\in E(Q)$.
Since $2 \cdot (\Delta/2)^2\cdot \mu n/\Delta   =  \mu N$, the colouring is $\mu N$-bounded.
Thus, by Theorem~\ref{mainthm}, $Q$ has a rainbow perfect matching $M$.

Consider the subgraph $H=(V(H), E(H))$ of $G$ with $V(H)=V(G)$ and edge set 
$$
E(H) = \{ vw \in E(G) :\, \text{there exist $1\leq i\leq \Delta/2<j\leq \Delta$ such that }u_{v,i} u_{w,j} \in E(M)\}\;.
$$
We claim that $H$ is a rainbow $\Delta$-factor of $G$. Since $u_{v,i} \in V(Q)$ for every $i \in [\Delta]$ and $M$ is a perfect matching of $Q$, we have $d_H(v) = \Delta$. Since $u_{v,i} u_{v,j} \not \in E(Q)$ for every $i,j\in [\Delta]$, $H$ has no self loops. Finally, since $M$ is a rainbow perfect matching of $Q$, and by definition of the colouring $\chi$, $H$ has no multiple edges and each colour in $\chi$ appears at most once in $M$. Thus, $H$ is a simple rainbow $\Delta$-regular spanning subgraph of $G$.
\end{proof}

Our second corollary concerns bipartite subgraphs of graphs with large minimum degree. Consider two graphs $G$ and $H$ on $n$ vertices  with $\Delta(H)\leq \Delta$. The Bollob\'as-Eldridge-Catlin conjecture~\cite{bollobas1978packings,catlin1974subgraphs}, states that if $\delta(G)\geq\left(1-1/(\Delta+1)\right)n -1/(\Delta+1)$, then $G$ contains a copy of $H$. Sauer and Spencer~\cite{sauerspencer} showed that the conjecture holds if $\delta(G)\geq \left(1-1/2\Delta\right) n -1$. 
This result has been improved for large values of $\Delta$~\cite{kaul2008graph}.
The existence of rainbow copies of $H$ in $k$-bounded edge colourings of $K_n$ was studied in~\cite{bottcher2012properly}, provided that $k=O(n/\Delta^2)$. In~\cite{sudakov2017properly}, it was observed that similar techniques allow to replace $K_n$ by a graph $G$ with $\delta(G)\geq \left(1-c/\Delta\right) n$, for a sufficiently small constant $c>0$.

Our last result partially extends the result in~\cite{bottcher2012properly} at the Sauer-Spencer minimum degree threshold.
\begin{theorem}
\label{ss_corr}
For every $\gamma>0$ there exists $\mu>0$ such that for every $\Delta\in \mathbb{N}$ there exists $n_0\in \mathbb{N}$ such that for every even $n\geq n_0$ the following holds.
Let $G$ be a graph on $n$ vertices with $\delta(G) \geq \left(1-1/2\Delta+\gamma\right)n$ and let $H$ be a balanced bipartite graph on $n$ vertices with $\Delta(H) \leq \Delta$, then any proper $(\mu n/\Delta^2)$-bounded edge colouring of $G$ contains a rainbow copy of $H$.
\end{theorem}
Sudakov and Volec~\cite{sudakov2017properly} showed that there exist a graph $H$ with maximum degree at most $\Delta$ and a $3.9 n /\Delta^2$-bounded edge colouring of $K_n$ which does not contain a rainbow copy of $H$. Therefore this theorem is also tight up to constant factors.
\begin{proof}
By Lemma 2.3 in~\cite{alonbipsg} there is a balanced bipartite spanning subgraph $G'=(A \cup B, E)$ of $G$ with minimum degree  $\delta(G')\geq (1 - 1/2\Delta +\gamma/2)m$, where $2m=n$.
By Theorem 3.5 in~\cite{bippacking}, the minimum degree condition ensures the existence of a subgraph $J$ of $G'$ that is isomorphic to $H$.
For each $a \in A$, let $N_a = \{ b \in B : ab \in E(J) \}$ denote the neighbourhood of $a$ in $J$.
Construct an auxiliary bipartite graph $Q=(V(Q), E(Q))$. The vertex set is $V(Q)=A \cup \Gamma$, where $\Gamma = \{\{N_a : a \in A \}\}$ as a multiset. 
The edge set is defined as  
$$
E(Q) = \{ a_1N_{a_2} :\, N_{a_2} \subseteq N_G(a_1) \}\;.
$$
Note that $Q$ is a balanced bipartite graph on $2m$ vertices. We first show that $\delta(Q)\geq (1/2 + \eps)m$, where $\eps=\gamma/2$. 
For each $a \in A$, there at most $(1/2\Delta - \eps)m$ vertices $b \in B$ such that $ab\notin E(G)$. 
Since $\Delta(J)\leq \Delta$, for each $b\in B$ there exist at most $\Delta$ vertices $a'\in A$ such that $b\in N_{a'}$. Thus, there are at most $(1/2 - \eps \Delta)m$ vertices $a'\in A$ such that $N_{a'}$ is not included in $N_G(a)$. In particular, we have $d_Q(a) \geq (1/2 + \eps)m$.
For each $a\in A$, we have $N_Q(N_a) = \cap_{b \in N_a} N_G(b)$.
So, by inclusion-exclusion, 
$$
d_Q(N_a) = \vert \cap_{b \in N_a} N_G(b) \vert \geq m - \sum_{b \in N_a} (m - \vert N_G(b) \vert) \geq  (1/2 + \eps)m\;.
$$
Hence, $\delta(Q)\geq (1/2 + \eps)n$.


Let $\chi$ be a proper $\mu n$-bounded edge colouring of $G$. Consider the following system of conflicts,
$$
\cF_Q = \{\{a_1N_{a_1'}, a_2N_{a_2'}\} : \exists\, x,y \in E(G) \text{ with } \chi(x)=\chi(y)\text{ and } \{x,y\} \subseteq E_G({a_1},N_{a_1'}) \cup E_G(a_2,N_{a_2'}) \}\;.
$$
Fix an edge  $a_1N_{a'_1} \in E(Q)$. 
For each $b_1 \in N_{a_1'}$, there are at most $\mu n/\Delta^2$ edges $a_2b_2$ with $\chi(a_2 b_2)= \chi(a_1b_1)$. Again, since $\Delta(J)\leq J$, $b_2$ is in at most $\Delta$ neighbourhoods $\Gamma_{a_2'}$. So, for each $b\in N_{a_1'}$, there are most $\mu n/\Delta$ edges $a_2N_{a_2'}$ conflicting with $a_1 N_{a_1'}$. Since $|N_{a_1'}|\leq \Delta$, the total number of conflicts involving edge $a_1 N_{a_1'}$ is at most $\mu n=2\mu m$. So $\cF$ is $2\mu m$-bounded.

We can apply Theorem \ref{conflicts} to the balanced bipartite graph $Q$ and the system of conflicts $\cF_Q$ to deduce the existence of a $\cF_Q$-conflict-free perfect matching $M$ in $Q$.
Define the subgraph $R=(V(R),E(R))$ of $G$ as follows. The vertex set is $V(R)=V(G)$ and edge set is
$$
E(R)=\{ab\in E(G):\,  \text{there exist $a'\in A$ such that }aN_{a'}\in E(M)\text{ and }b\in N_{a'}\}\;.
$$
We claim that $R$ is a rainbow subgraph of $G$ isomorphic to $H$.
Consider a bijective map $f: V(G) \to V(G)$, such that $f(u)=v$ if and only if $u N_v\in M$ for $u \in A$ and $f$ is the identity map on $B$.
We claim that $f$ is an isomorphism from $R$ to $J$.
To see this, first observe that $f$ is an automorphism of $V(G)$.
Now, consider an edge $ab \in E(R)$ and note that $f(a)f(b) = f(a)b$ where $f(a)$ is such that $a N_{f(a)} \in M$.
As $M$ is a matching, there is only one choice $N_{a'}$ such that $a N_{a'} \in E(M)$, implying that $a' = f(a)$.
By definition of $E(R)$, we have that $b \in N_{f(a)}$, so $f(a)b = f(a)f(b) \in E(J)$.
Similarly, one can show that for all edges $ab \in E(J)$, $f^{-1}(a)f^{-1}(b)= f^{-1}(a) b\in E(R)$. Thus $f$ is an isomorphism between $R$ and $J$, and since $J$ is isomorphic to $H$, so is $R$.

Finally, suppose for contradiction that there exist $x,y\in E(R)$ with $\chi(x)=\chi(y)$. If $x=a_1b_1$ and $y=a_2b_2$, let $a_1',a_2'\in A$ be such that $a_1N_{a_1'}, a_2'N_{a_2'} \in E(M)$. Then, as $x,y\in E(R)$, we have $b_1 \in N_{a_1'}$ and $b_2 \in N_{a_2'}$, implying that $a_1N_{a_1'}$ and $  a_2'N_{a_2'}$ conflict under $\cF_Q$. This is a contradiction as $M$ is a $\cF_Q$-conflict-free perfect matching. So $R$ is rainbow.
\end{proof}

\section{Further remarks}\label{sec:fur_rem}
We conclude the paper with a number of remarks and open questions.
\begin{itemize}
\item[1)] The condition on the minimum degree in Theorem~\ref{mainthm} is best possible. However, the value of $\mu$ that follows from our proof is far from being optimal. In Section~\ref{sec:final}, we showed that the statement is not true if $\mu>1/4$. Obtaining the best possible value for $\mu$ is a difficult problem, since it would imply a minimum degree version of the Ryser-Brualdi-Stein conjecture, which is wide open.
\item[2)] We believe that the statement of Theorem~\ref{mainthm} should also hold for system of conflicts. The only obstacle in our proof is Lemma~\ref{boxes}, which, in its current form, cannot be adapted to deal with conflicts instead of colours.

\item[3)] As shown in Section~\ref{sec:appl}, the methods presented in this paper are of potential interest to embed other conflict-free spanning structures in graphs with large minimum degree, beyond perfect matchings. A natural candidate is to embed Hamiltonian cycles in Dirac graphs. Krivelevich et al.t~\cite{KLScompatible} proved the existence of $\cF$-conflict-free Hamiltonian cycles in Dirac graphs, provided that the conflicts in $\cF$ are local. Their proof is substantially different from ours and relies on P\'{o}sa rotations. 

\item[4)] Lu and Sz\'ekely generalised the idea of system of conflicts to include, not only unordered pairs of edges, but sets of edges of any size~\cite{ls2007}. Under some sparsity conditions on the set of conflicts, they proved the existence of conflict-free perfect matchings in $K_{n,n}$. Our results can be seen as a first step towards extending the Lu-Sz\'ekely framework to Dirac graphs.

\item[5)] Csaba~\cite{bipBEC} proved the Bollob\'as-Eldridge-Catlin conjecture for embedding bipartite graphs of maximum degree $\Delta$ into any graph $G$ of minimum degree at least $(1-\beta)(1-1/(\Delta+1)) n$ for some $\beta>0$. It would be of interest to determine whether a form of Theorem~\ref{ss_corr} holds in this setting, since it does not follow as a direct consequence of Theorem~\ref{mainthm}. 


\end{itemize}

\medskip

\section{Acknowledgements} 

The authors want to thank Daniela K\"{u}hn, Allan Lo, Deryk Osthus and Benny Sudakov for fruitful discussions and remarks on the topic.

\bibliographystyle{plain}
{\small \bibliography{RPMbibliography}}

\begin{thebibliography}{10}

\bibitem{alonbipsg}
N.~Alon and E.~Fischer.
\newblock 2-factors in dense graphs.
\newblock {\em Discrete Mathematics}, 152(1):13 -- 23, 1996.

\bibitem{bollobas1978packings}
B.~Bollob{\'a}s and S.~E. Eldridge.
\newblock Packings of graphs and applications to computational complexity.
\newblock {\em Journal of Combinatorial Theory, Series B}, 25(2):105--124,
  1978.

\bibitem{bottcher2012properly}
J.~B{\"o}ttcher, Y.~Kohayakawa, and A.~Procacci.
\newblock Properly coloured copies and rainbow copies of large graphs with
  small maximum degree.
\newblock {\em Random Structures \& Algorithms}, 40(4):425--436, 2012.

\bibitem{brualdi1991combinatorial}
R.~A. Brualdi and H.~J. Ryser.
\newblock {\em Combinatorial matrix theory}, volume~39.
\newblock Cambridge University Press, 1991.

\bibitem{catlin1974subgraphs}
P.~A. Catlin.
\newblock Subgraphs of graphs, {I}.
\newblock {\em Discrete Mathematics}, 10(2):225--233, 1974.

\bibitem{bipBEC}
B.~Csaba.
\newblock On the {B}ollob{\'a}s--{E}ldridge conjecture for bipartite graphs.
\newblock {\em Combinatorics, Probability and Computing}, 16(5):661--691, 2007.

\bibitem{millie}
L.~DeBiasio, D.~K{\"u}hn, T.~Molla, D.~Osthus, and A.~Taylor.
\newblock Arbitrary orientations of {H}amilton cycles in digraphs.
\newblock {\em SIAM Journal on Discrete Mathematics}, 29(3):1553--1584, 2015.

\bibitem{l4es}
P.~Erd{\H{o}}s and J.~Spencer.
\newblock Lopsided {L}ov{\'a}sz local lemma and {L}atin transversals.
\newblock {\em Discrete Applied Mathematics}, 30(2-3):151--154, 1991.

\bibitem{edelta}
A.~B. Evans.
\newblock Latin squares without orthogonal mates.
\newblock {\em Designs, Codes and Cryptography}, 40(1):121--130, 2006.

\bibitem{bippacking}
J.~L. Fouquet and A.~P. Wojda.
\newblock Mutual placement of bipartite graphs.
\newblock {\em Discrete Mathematics}, 121(1):85 -- 92, 1993.

\bibitem{hs2008}
P.~Hatami and P.~W. Shor.
\newblock A lower bound for the length of a partial transversal in a {L}atin
  square.
\newblock {\em J. Combin. Theory Ser. A}, 115(7):1103--1113, 2008.

\bibitem{JLRchernoff}
S.~Janson, T.~Luczak, and A.~Rucinski.
\newblock {\em Random graphs}, volume~45.
\newblock John Wiley \& Sons, 2011.

\bibitem{kfactorexists}
P.~Katerinis.
\newblock Minimum degree of a graph and the existence of $k$-factors.
\newblock {\em Proceedings of the Indian Academy of Sciences - Mathematical
  Sciences}, 94(2):123--127, 1985.

\bibitem{kaul2008graph}
H.~Kaul, A.~Kostochka, and G.~Yu.
\newblock On a graph packing conjecture by {B}ollob{\'a}s, {E}ldridge and
  {C}atlin.
\newblock {\em Combinatorica}, 28(4):469--485, 2008.

\bibitem{KLScompatible}
M.~Krivelevich, C.~Lee, and B.~Sudakov.
\newblock Compatible {H}amilton cycles in {D}irac graphs.
\newblock {\em Combinatorica}, 37(4):697--732, 2017.

\bibitem{REepsKLOS}
D.~K{\"u}hn, A.~Lo, D.~Osthus, and K.~Staden.
\newblock The robust component structure of dense regular graphs and
  applications.
\newblock {\em Proceedings of the London Mathematical Society}, 110(1):19--56,
  2014.

\bibitem{kuhn2014hamilton}
D.~K{\"u}hn and D.~Osthus.
\newblock Hamilton cycles in graphs and hypergraphs: an extremal perspective.
\newblock {\em Proceedings of the International Congress of Mathematicians
  2014}, 4:381--406, 2014.

\bibitem{DDT}
D.~K{\"u}hn, D.~Osthus, and A.~Treglown.
\newblock Hamiltonian degree sequences in digraphs.
\newblock {\em Journal of Combinatorial Theory, Series B}, 100(4):367--380,
  2010.

\bibitem{ls2007}
L.~Lu and L.~Sz{\'e}kely.
\newblock Using {L}ov{\'a}sz local lemma in the space of random injections.
\newblock {\em the electronic journal of combinatorics}, 14(1):R63, 2007.

\bibitem{mrazuma}
M.~Molloy and B.~Reed.
\newblock {\em Graph colouring and the probabilistic method}, volume~23.
\newblock Springer Science \& Business Media, 2013.

\bibitem{mmHC}
J.~Moon and L.~Moser.
\newblock On {H}amiltonian bipartite graphs.
\newblock {\em Israel Journal of Mathematics}, 1(3):163--165, 1963.

\bibitem{bs2017}
A.~Pokrovskiy and B.~Sudakov.
\newblock A counterexample to {S}tein's equi-n-square conjecture.
\newblock {\em arXiv:1711.00429}, 2017.

\bibitem{r1967}
H.~J. Ryser.
\newblock Neuere probleme der kombinatorik.
\newblock {\em Vortrageber Kombinatorik, Oberwolfach}, 1967.

\bibitem{sauerspencer}
N.~Sauer and J.~Spencer.
\newblock {Edge disjoint placement of graphs}.
\newblock {\em Journal of Combinatorial Theory, Series B}, 25(3):295--302,
  1978.

\bibitem{s1975}
S.~K. Stein.
\newblock Transversals of {L}atin squares and their generalizations.
\newblock {\em Pacific J. Math.}, 59(2):567--575, 1975.

\bibitem{sudakov2017properly}
B.~Sudakov and J.~Volec.
\newblock Properly colored and rainbow copies of graphs with few cherries.
\newblock {\em Journal of Combinatorial Theory, Series B}, 122:391--416, 2017.

\bibitem{wwdelta}
I.~M. Wanless and B.~S. Webb.
\newblock The existence of {L}atin squares without orthogonal mates.
\newblock {\em Designs, Codes and Cryptography}, 40(1):131--135, 2006.

\end{thebibliography}
\end{document}